\documentclass[usletter, 10pt, conference]{IEEEtran}

\usepackage{here}
\usepackage{graphicx}          
\usepackage{subfigure}                               

\usepackage{amssymb}
\usepackage{amsmath}
\usepackage{graphicx}
\usepackage{amsfonts}
\usepackage{cite}
\DeclareMathOperator{\diag}{diag}
\usepackage{dsfont,accents}
\usepackage{url}

\everymath{\displaystyle}

\newtheorem{theorem}{Theorem}

\newtheorem{proposition}[theorem]{Proposition}
\newtheorem{lemma}[theorem]{Lemma}

\newtheorem{property}[theorem]{Property}

\newcommand{\mendth}{\hfill \ensuremath{\vartriangle}}

\DeclareMathOperator*{\He}{Sym}

\DeclareMathOperator{\eps}{\varepsilon}

\DeclareMathOperator{\E}{\text{\textbf{E}}}

\title{Computer control of gene expression: Robust setpoint tracking of protein mean and variance using integral feedback}

\author{Corentin Briat and Mustafa Khammash
\thanks{Corentin Briat and Mustafa Khammash are with the Department of Biosystems Science and Engineering {(D-BSSE)}, Swiss Federal Institute of Technology--Z\"{u}rich {(ETH--Z)}, Mattenstrasse 26, 4058 Basel, Switzerland; email: {\tt  \{corentin.briat,mustafa.khammash\}@bsse.ethz.ch}; url: \protect\url{http://www.bsse.ethz.ch/ctsb}, \protect\url{http://www.briat.info}, \protect\url{http://www.bsse.ethz.ch/research/Professors/khammash\textunderscore cv}}}

\begin{document}
\sloppy
\maketitle

\begin{abstract}
Protein mean and variance levels in a simple stochastic gene expression circuit are controlled using proportional integral feedback. It is shown that the protein mean level can be globally and robustly tracked to any desired value using a simple PI controller that satisfies explicit sufficient conditions. Controlling both the mean and variance on the other hand requires the use of an additional control input, chosen here as the mRNA degradation rate. Local robust tracking of mean and variance is proved to be achievable using multivariable PI control, provided that the reference point satisfies necessary conditions imposed by the system. Even more importantly, it is shown that there exist PI controllers that locally, robustly and simultaneously stabilize all the equilibrium points inside the admissible region. Simulation examples illustrate the results.
\end{abstract}

\begin{keywords}
Gene expression network; moment control; PI control; absolute stability.
\end{keywords}

\section{Introduction}

Synthetic biology is an emergent field of biology/biotechnology in which living cells are genetically modified in order to achieve new functions. Biological circuits of interacting genes and proteins have been successfully introduced into living cells to implement various functioning modules such as oscillators \cite{Elowitz:05,Hasty:08}, switches \cite{Collins:00}, logic gates \cite{Wang:11} among others. Inspired by electronic circuit design, synthetic biologists aim to design functioning modules that can be put together in various configurations to build biological devices with designed function.  In spite of many successes, circuits with more than a few interacting genes can rarely built reliably and predictably. For this reason, when it comes to regulation, building effective biosynthetic control circuits remains a big challenge.

Recently it has been shown that more complex regulation becomes possible by moving control functions outside the cells and relegating these to a digital computer \cite{Khammash:11}. This `in-silico' regulation of living cells also allows the implementation of controllers with more accuracy and higher speeds than what is possible within the environment of the cell.  Using synthetic biology methods, yeast cells were genetically engineered so that their gene expression was responsive to light signals. At the same time, flow cytometry and microscopy methods were used to measure and quantify the resulting protein expression levels in real time. These measurements were then fed to a computer control system that used a Kalman filter/MPC control configuration to generate a control signal consisting of a train of light pulses that drove gene expression. Using this setup, feedback control of gene expression in living cells was successfully demonstrated experimentally \cite{Khammash:11}.

Preliminary results on moment control of reaction networks using PI control laws have been obtained in \cite{Klavins:10} where it is shown that such controllers can achieve the desired objectives for certain simple reaction networks. The problem addressed in this paper is slightly different and directly inspired from the relevant gene expression network considered in \cite{Khammash:11}. In the current work, we use a stochastic model of gene expression to explore the feasibility of using simple controllers and continuous (non-pulsed) control inputs to achieve effective genetic control of protein mean and variance. The problems addressed in the paper are beyond those of \cite{Klavins:10}  since existence and characterization of positive PI controllers, accounting for the presence of nonlinearities, are discussed. It is shown that local regulation of protein mean levels to any desired value is achievable by acting on the DNA transcription rate using proportional-integral feedback. The PI controller is also shown to be locally robust and exact regions in the controller parameter space that maintain this local robustness are derived. The positivity requirement for control, ignored in \cite{Klavins:10}, introduces a static nonlinearity in the feedback system and must be considered to rigorously characterize global stability of the controlled system. Using absolute stability theory \cite{Lure:44,Liberzon:02} and the Popov criterion \cite{Popov:61}, it is proved that global asymptotic stability is achievable when the controller gains satisfy very mild sufficient conditions.

Next, we show that using a second control input that controls mRNA degradation, a multivariable PI feedback controller can be designed so that any desired protein mean and variance setpoints that lie within a certain admissibility region defined by the system can be locally and robustly tracked. More importantly, it is also shown that there exist common multivariable PI controllers that locally and robustly stabilizes all the equilibrium points in the admissible region. Numerical simulations finally demonstrate the effectiveness of the designed genetic control systems.

\textit{Outline:} In Section \ref{sec:prob}, the general framework and the main problem, i.e. the control of the moments related to the master equation, of the paper are introduced. Sections \ref{sec:mean} and \ref{sec:var} respectively address the problems of controlling the mean and the variance of the number of proteins in a simple gene expression network. Examples are treated in Section \ref{sec:ex}.

\textit{Notations:} The notation is standard. Given a random variable $X$, its expectation is denoted by $\E[X]$. For a square matrix $M$, $\He[M]$ stands for the sum $M+M^T$. Given a vector $v\in\mathbb{R}^n$, the notation $\diag(v)$ stands for a diagonal matrix having the elements of $v$ as diagonal entries.

\section{Problem Statement}\label{sec:prob}

\subsection{General framework}

Let us start with the general stochastic formulation where $N$ molecular species $S_1,\ldots,S_N$ interact with each others through $M$ reaction channels $R_1,\ldots,R_M$. Assuming homogeneous mixing and thermal equilibrium, the time evolution of the random variables $X_1(t),\ldots,X_N(t)$ associated with the population of each species can be described by the so-called Chemical Master Equation (CME), or Forward Kolmogorov equation, given by
\begin{equation*}
  \dot{P}(\varkappa,t)=\sum_{k=1}^M\left[w_k(\varkappa-s_k)P(\varkappa-s_k,t)-w_k(\varkappa)P(\varkappa,t)\right]
\end{equation*}
where $s_k$ is the stoichiometry vector associated with reaction $R_k$ and $w_k$ the propensity function capturing the rate of the reaction $R_k$. The variable $\varkappa$ is the state-variable and $P(\varkappa,t)$ denotes the probability to be in state $\varkappa$ at time $t$.

Based on the CME, dynamical expressions for the first- and second-order moments may be easily derived and are given by
\begin{equation}\label{eq:moments}
\hspace{-3mm}\begin{array}{lcl}
    \dfrac{d \E[X]}{dt}&=&S\E[w(X)],\\
    \dfrac{d \E[XX^T]}{dt}&=&S\E[w(X)X^T]+\E[w(X)X^T]^TS^T\\
    &&+S\diag\{\E[w(X)]\}S^T
\end{array}
\end{equation}
where $S:=\begin{bmatrix}
  s_1 & \ldots & s_M
\end{bmatrix}\in\mathbb{R}^{N\times M}$ is the stoichiometry matrix and $w(X):=\begin{bmatrix}
  w_1^T & \ldots & w_M^T
\end{bmatrix}^T\in\mathbb{R}^{M}$ the propensity vector.

According to the structure of the propensity functions $w(X)$, the above set of equations may suffer from well-posedness and closedness problems. This is explained by the fact that computing a moment of a certain order may require the knowledge of higher order moments. It is therefore not possible, in this case, to describe the evolution of the the first moments by a finite set of ordinary differential equations.  Some approximation schemes have been proposed in the literature to overcome this problem, see e.g. \cite{Singh:07,Gillepsie:09,Milner:11}, where it is proposed to approximate higher order moments as functions of lower order ones.

\subsection{Affine propensity case and gene expression}

In the affine propensity case, i.e. $w(X)=WX+w_0$, $W\in\mathbb{R}^{M\times N}$, $w_0\in\mathbb{R}^M$, things turn out to be much more convenient. In this case, the moments equations can be reformulated as
\begin{equation}\label{eq:moments}
\begin{array}{rcl}
    \dfrac{d \E[X]}{dt}&=&SW\E[X]+Sw_0,\\
    \dfrac{d\Sigma}{dt}&=&\He[SW\Sigma]+S\diag(W\E[X]+w_0)S^T
\end{array}
\end{equation}
where $\Sigma:=\E[(X-\E[X])(X-\E[X])^T]$ is the covariance matrix. It is immediate to see that when the matrix $SW$ is Hurwitz, the mean and variance trajectories exponentially converge to the equilibrium points $\bar{X}$ and $\bar{\Sigma}$ given by
\begin{equation}
  \hspace{-3mm}\begin{array}{lcl}
   \bar{X}&=& -(SW)^{-1}Sw_0\\
   \bar{\Sigma}&=&\int_0^\infty e^{W^TS^Ts}S\diag(W\bar{X}+w_0)S^Te^{SWs}ds.
  \end{array}
\end{equation}

\begin{figure}[H]
\begin{minipage}[c]{0.52\linewidth}
\centering
 \includegraphics[width=\textwidth]{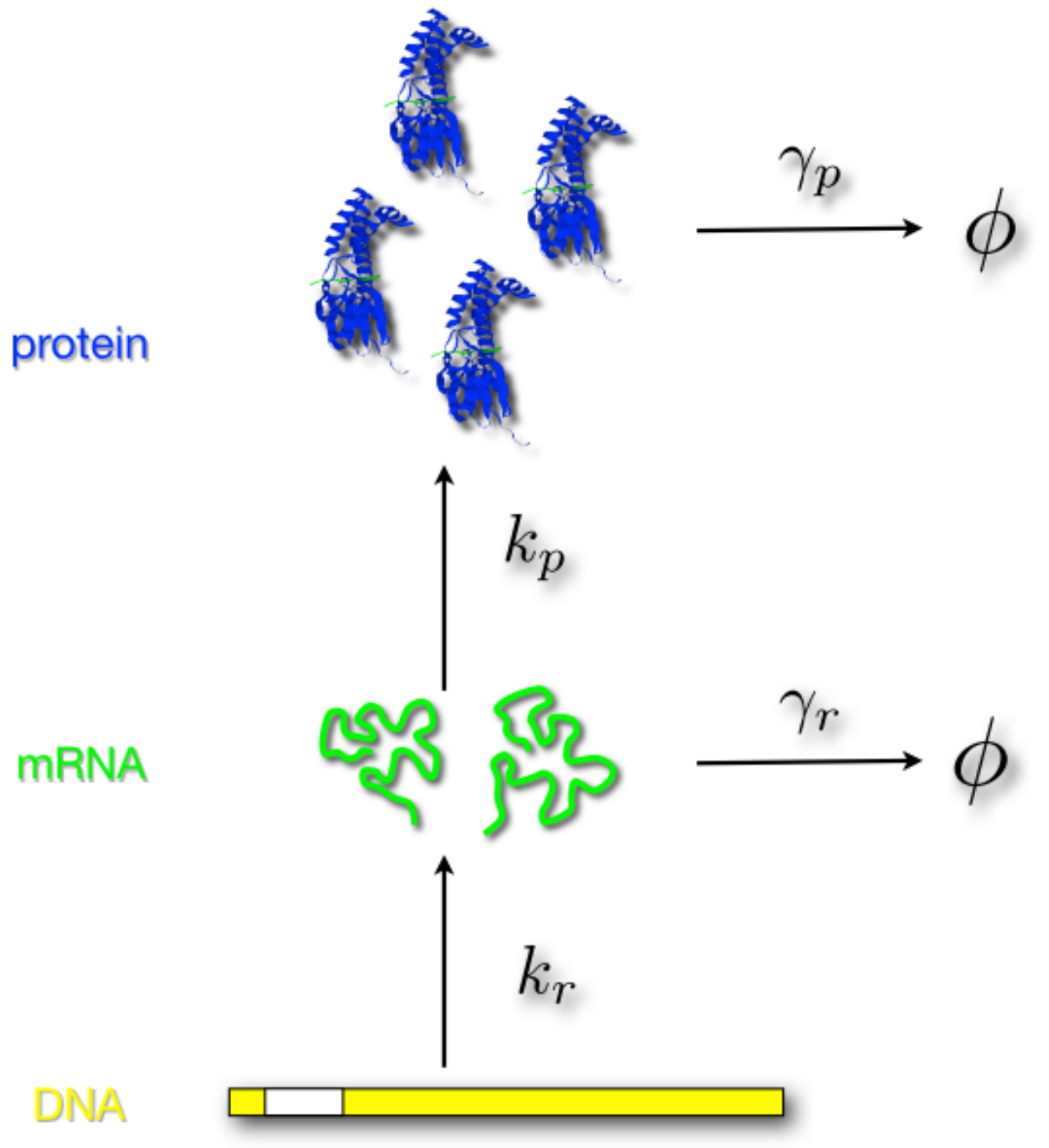}
\end{minipage}
\hfill
\begin{minipage}[l]{0.47\linewidth}
\centering
 {\footnotesize $\begin{array}{l}
   R_1:\phi\stackrel{k_r}{\longrightarrow}\text{mRNA}\\
   R_2:\text{mRNA}\stackrel{\gamma_r}{\longrightarrow}\phi\\
   R_3:\text{mRNA}\stackrel{k_p}{\longrightarrow}\text{protein+mRNA}\\
   R_4:\text{protein}\stackrel{\gamma_p}{\longrightarrow}\phi
 \end{array}$}
 \end{minipage}
 \caption{Simple gene expression network.}\label{fig:gene}
\end{figure}

Considering now the simple gene expression circuit depicted in Fig. \ref{fig:gene} which involves two species, i.e. mRNA ($S_1$) and protein ($S_2$) species, and four reaction channels $R_1,\ldots,R_4$. Let $X_1(t)$ and $X_2(t)$ be random variables describing the number of mRNA and protein molecules respectively. The stoichiometry matrix associated with the gene expression network is given by
\begin{equation}
    S=\begin{bmatrix}
      1 & -1 & 0 & 0\\
      0 & 0 & 1 & -1
    \end{bmatrix}
\end{equation}
and the (affine) propensity vector by
\begin{equation}
    w(X)=\begin{bmatrix}
      k_r & \gamma_rX_1 &  k_pX_1 & \gamma_pX_2
    \end{bmatrix}^T.
\end{equation}
In vector form, equations (\ref{eq:moments}) rewrite
\begin{equation}\label{eq:mainsyst}
\left[\begin{array}{c}
    \dot{x}_1(t)\\
    \dot{x}_2(t)\\
    \hline
    \dot{x}_{3}(t)\\
    \dot{x}_{4}(t)\\
    \dot{x}_{5}(t)
  \end{array}\right]=\left[\begin{array}{c|c}
    A_{ee} & 0\\
    \hline
    A_{\sigma\sigma} & A_{\sigma e}
  \end{array}\right]\left[\begin{array}{c}
    x_1(t)\\
    x_2(t)\\
    \hline
     x_{3}(t)\\
    x_{4}(t)\\
    x_{5}(t)
  \end{array}\right]+\left[\begin{array}{c}
    B_e\\
    \hline
    B_\sigma
  \end{array}\right]k_r
\end{equation}
where the state variables are defined by
\begin{equation*}
  \begin{bmatrix}
    x_1\\
    x_2
  \end{bmatrix}:=\E[X]\ \text{and}\ \begin{bmatrix}
    x_3 & x_4\\
    x_4 & x_5
  \end{bmatrix}:=\Sigma
\end{equation*}
and the system matrices by
\begin{equation}\label{eq:mainsystm}
\begin{array}{lcl}
   A_{ee}&=&\begin{bmatrix}
    -\gamma_r & 0\\
    k_p & -\gamma_p
  \end{bmatrix},\ A_{\sigma e}=\begin{bmatrix}
    \gamma_r & 0\\
    0 & 0\\
    k_p & \gamma_p
  \end{bmatrix},\ B_e=\begin{bmatrix}
    1\\0
  \end{bmatrix},\\
  A_{\sigma\sigma}&=&\begin{bmatrix}
    -2\gamma_r & 0 & 0\\
    k_p & -(\gamma_r+\gamma_p) & 0\\
    0 & 2k_p & -2\gamma_p
  \end{bmatrix},\ B_\sigma=\begin{bmatrix}
    1\\
    0\\
    0
  \end{bmatrix},
\end{array}
\end{equation}
where $k_r>0$ is the transcription rate of DNA into mRNA, $\gamma_r>0$ is the degradation rate of mRNA, $k_p>0$ is the translation rate of mRNA into protein and $\gamma_p>0$ is the degradation rate of the protein.
%
%
\begin{property}
System (\ref{eq:mainsyst})-(\ref{eq:mainsystm}) is asymptotically stable.
\end{property}

\begin{property}
  The equilibrium point of system (\ref{eq:mainsyst})-(\ref{eq:mainsystm}) is given by
\begin{equation}
\begin{array}{lcl}
         x_1^*&=&\dfrac{k_r}{\gamma_r},\quad x_2^*=\dfrac{k_pk_r}{\gamma_p\gamma_r},\ \quad  x_3^*= \dfrac{k_r}{\gamma_r},\\
         x_4^*&=&\dfrac{k_pk_r}{\gamma_r(\gamma_p +\gamma_r)},\
         x_5^*=\dfrac{k_pk_r(\gamma_p + k_p+\gamma_r)}{\gamma_p\gamma_r(\gamma_p + \gamma_r)}.
\end{array}
\end{equation}
\end{property}

\section{Mean control}\label{sec:mean}

The objective of the current section is to give a clear picture of the mean control of the number of proteins using a simple \emph{positive PI controller}, i.e. a PI controller generating nonnegative control inputs. The considered control input is the transcription rate $k_r$ which has been shown to be possibly externally actuated using, for instance, light-induced transcription \cite{Khammash:11}. It is shown in this section that a positive PI control law allows to achieve global and robust output tracking of the mean number of proteins.

\subsection{Preliminaries}

The considered system consists of the interconnection of the following restriction of system (\ref{eq:mainsyst}):
\begin{equation}\label{eq:musyst}
  \begin{bmatrix}
    \dot{x}_1(t)\\
    \dot{x}_2(t)
  \end{bmatrix}=A_{ee}\begin{bmatrix}
    x_1(t)\\
    x_2(t)
  \end{bmatrix}+B_eu(t)
\end{equation}
and the positive PI control law
\begin{equation}\label{eq:pi1}
  u(t)=\varphi\left(k_1(\mu_*-x_2(t))+k_2\int_0^t[\mu_*-x_2(s)]ds\right)
\end{equation}
where $\mu_*$ is the mean number of protein to track and the scalars $k_1,k_2$ the gains of the controller. The on-off nonlinearity $\varphi(u):=\max\{0,u\}$, see e.g.  \cite{Goncalves:07}, is considered in order to impose the positivity of the control input, i.e. the positivity of $k_r\equiv u$.

\begin{property}
  Given a constant reference $\mu_*\ge0$, the equilibrium point of the system (\ref{eq:musyst})-(\ref{eq:pi1}) is given by $x_2^*=\mu_*$ and
  \begin{equation}\label{eq:mupt}
     x_1^*=\dfrac{\mu_*\gamma_p}{k_p},\ u^*=\dfrac{\mu_*\gamma_p\gamma_r}{k_p},\ I^*=\frac{u^*}{k_2}
  \end{equation}
  where $I^*$ is the equilibrium value of the integral term.
\end{property}
As expected, the presence of the integrator allows to rule out any steady state error in the constant reference case.

\subsection{Local stabilizability, stabilization and output tracking}

Since the equilibrium control input $u^*$ and the reference $\mu_*$ are simultaneously positive, the nonlinearity $\varphi(\cdot)$ is not active in a sufficiently small neighborhood of the equilibrium point (\ref{eq:mupt}). The on-off nonlinearity can hence be locally ignored and the local analysis performed on the corresponding linear system. Assuming first that the system parameters $k_p,\gamma_p,\gamma_r$ are exactly known, the following result on local nominal stabilizability and stabilization can be obtained:
\begin{lemma}
  Given system parameters $k_p,\gamma_p,\gamma_r>0$, the system (\ref{eq:musyst}) is locally stabilizable using the control law (\ref{eq:pi1}).  Moreover, the equilibrium point (\ref{eq:mupt}) of the closed-loop system (\ref{eq:musyst})-(\ref{eq:pi1}) is locally asymptotically (exponentially) stable if and only if the conditions
  \begin{equation}\label{eq:nomcond}
   \begin{array}{rclcrcl}
     k_1&>&\dfrac{k_2}{\gamma_p+\gamma_r}-\dfrac{\gamma_p\gamma_r}{k_p}&\text{and}& k_2&>&0
   \end{array}
  \end{equation}
  hold.\mendth
\end{lemma}
\begin{proof}
  The local augmented system gathering the linear dynamics of the system (\ref{eq:musyst}) and the controller (\ref{eq:pi1}) is given by
  \begin{equation}\label{eq:aug}
\begin{bmatrix}
  \dot{x}_1(t)\\
  \dot{x}_2(t)\\
  \dot{I}(t)
\end{bmatrix}=\begin{bmatrix}
    -\gamma_r & -k_1 & k_2\\
    k_p & -\gamma_p & 0\\
    0 & -1 & 0
  \end{bmatrix}\begin{bmatrix}
  x_1(t)\\
  x_2(t)\\
  I(t)
\end{bmatrix}+\begin{bmatrix}
  k_1\\0\\1
\end{bmatrix}\mu_*
  \end{equation}
where $I$ is the integrator state of the controller. Local stabilizability is then equivalent to the existence of a pair $(k_1,k_2)\in\mathbb{R}^2$ such that the state matrix of the augmented system (\ref{eq:aug}) is Hurwitz, i.e. has poles in the open left-half plane. The Routh-Hurwitz criterion yields the conditions (\ref{eq:nomcond}) that define a nonempty subset of the plane $(k_1,k_2)$. System (\ref{eq:musyst}) is hence locally stabilizable using the PI control law (\ref{eq:pi1}) for any triplet of parameter values ${(k_p,\gamma_r,\gamma_p)\in\mathbb{R}_{>0}^3}$. As a consequence, the closed-loop system is locally asymptotically stable when the control parameters are located inside the stability region defined by the conditions (\ref{eq:nomcond}).
\end{proof}

In order to extend the above result to the uncertain case, we assume here that the system parameters $(k_p,\gamma_p,\gamma_r)$ belong to the set
\begin{equation}\label{eq:paramset}
  \mathcal{P}_\mu:=(0,k_p^+]\times[\gamma_p^-,\infty)\times [\gamma_r^-,\infty)
\end{equation}
where the parameter bounds $k_p^+,\gamma_r^-$ and $\gamma_p^-$ are real positive numbers. We then obtain the following result:
\begin{lemma}\label{lem:robloc}
  The system (\ref{eq:musyst}) with uncertain constant parameters $(k_p,\gamma_p,\gamma_r)\in\mathcal{P}_\mu$ is robustly locally asymptotically (exponentially) stabilizable using the control law (\ref{eq:pi1}). Moreover, the equilibrium point (\ref{eq:mupt}) of the closed-loop system (\ref{eq:musyst})-(\ref{eq:pi1}) is locally robustly asymptotically (exponentially) stable if and only if the conditions
  \begin{equation}\label{eq:robcond}
   \begin{array}{rclcrcl}
     k_1&>&\dfrac{k_2}{\gamma_p^-+\gamma_r^-}-\dfrac{\gamma_r^-\gamma_p^-}{k_p^+}&\text{and}&     k_2&>&0
   \end{array}
  \end{equation}
  hold.\mendth
\end{lemma}
\begin{proof}
   Define the lower bound function for $k_1$ by $f(x,y,z):=\dfrac{k_2}{y+z}-\dfrac{yz}{x}$, $x,y,z,k_2>0$ and let ${\bar{f}:=\sup_{(x,y,z)\in\mathcal{P}_\mu}f(x,y,z)}$. Simple calculations show that $f(x,y,z)$ is increasing in $x$ and decreasing in $y,z$ over $(x,y,z)\in\mathcal{P}_\mu$. Hence, we have $\bar{f}=f(k_p^+,\gamma_p^-,\gamma_r^-)$ and $k_1>\bar{f}$ implies that $k_1>f(x,y,z)$ for all $(x,y,z)\in\mathcal{P}_\mu$. This concludes the proof.
\end{proof}

\subsection{Global stabilizability, stabilization and output tracking}

Local properties obtained in the previous section are generalized here to global ones.  Noting first that the nonlinear function $\varphi(\cdot)$ is time-invariant and belongs to the sector $[0,1]$, i.e. $0\le\varphi(x)/x\le 1$, $x\in\mathbb{R}$, stability can then be analyzed using absolute stability theory \cite{Khalil:02} and an extension of the Popov criterion \cite{Popov:61,Khalil:02} for marginally stable systems \cite{Jonsson:97,Fliegner:06}. We have the following result:
%
\begin{theorem}
   Given system parameters $k_p,\gamma_p,\gamma_r>0$, then the equilibrium point  (\ref{eq:mupt}) of the closed-loop system (\ref{eq:musyst})-(\ref{eq:pi1}) is globally asymptotically stable if the following conditions
   \begin{equation}\label{eq:djsqdj}
   \begin{array}{rclcrcl}
     k_2&>&0&\text{and}& k_1&>&\dfrac{k_2}{\gamma_p}
   \end{array}
   \end{equation}
   hold.\mendth
\end{theorem}
\begin{proof}
Following the absolute stability paradigm, the closed-loop system (\ref{eq:musyst})-(\ref{eq:pi1}) is rewritten as the interconnection of the marginally stable LTI system
\begin{equation}
  H(s)=\dfrac{k_p(k_1s+k_2)}{s(s+\gamma_r)(s+\gamma_p)}
\end{equation}
and the static nonlinearity $\varphi(\cdot)$. We assume in the following that $k_2>0$, which is a necessary condition for local asymptotic stability of the equilibrium point (\ref{eq:mupt}). Since $\lim_{s\to0}sH(s)=\frac{k_pk_2}{\gamma_r\gamma_p}>0$, then the Popov criterion \cite{Popov:61,Fliegner:06} can be applied. It states that the system (\ref{eq:musyst}) is absolutely stabilizable using controller (\ref{eq:pi1}) if there exist $(k_1,k_2)\in\mathbb{R}^2$ and $q\ge0$ such that the condition
\begin{equation}\label{eq:cost}
F(j\omega,q):=\Re\left[(1+qj\omega)H(j\omega)\right]>-1
\end{equation}
holds for all $\omega\in\mathbb{R}$. In order to check this condition, first rewrite $F(j\omega,q)$ as
\begin{equation}
  F(j\omega,q)=\dfrac{N_0(\omega)}{D(\omega)}+q\dfrac{N_1(\omega)}{D(\omega)}
\end{equation}
where
\begin{equation}
  \begin{array}{lcl}
        N_0(\omega)&=&k_p\left[k_1(\gamma_r\gamma_p-\omega^2)-k_2(\gamma_r+\gamma_p)\right]\\
        N_1(\omega)&=&k_p\left[k_1\omega^2(\gamma_r+\gamma_p)+k_2(\gamma_p\gamma_r-\omega^2)\right]\\
        D(\omega)&=&(\omega^2+\gamma_r^2)(\omega^2+\gamma_p^2).
  \end{array}
\end{equation}
Since $D(\omega)>0$ for all $\omega\in\mathbb{R}$, then the condition (\ref{eq:cost}) is equivalent to
\begin{equation}
  N_0(\omega)+q N_1(\omega)+D(\omega)>0
\end{equation}
for all $\omega\in\mathbb{R}$. Letting $\bar{\omega}:=\omega^2$, we get
\begin{equation}\label{eq:P}
  Z(\bar{\omega}):=\bar{\omega}^2+z_1(q)\bar{\omega}+z_0(q)>0
\end{equation}
for all $\bar{\omega}\in[0,\infty)$ where
\begin{equation*}
  \begin{array}{lcl}
    z_0(q)&=&\gamma_r^2\gamma_p^2+k_p\left[\gamma_r(\gamma_pk_1-k_2+\gamma_pk_2q)-\gamma_pk_2\right]\\
    z_1(q)&=&\gamma_p^2+\gamma_r^2+k_p\left[(k_1\gamma_p-k_2)q-k_1+k_1\gamma_r\right].
  \end{array}
\end{equation*}
The problem therefore essentially becomes a positivity analysis of the polynomial $Z(\bar{\omega})$ over $[0,\infty)$. By virtue of Descartes's rule of signs \cite{Henrici:88}, if there exists $q\ge0$ such that $z_0(q)>0$ and $z_1(q)>0$, then $Z$ does not have positive real zeros, and hence the Popov condition is verified. Since $k_2>0$, it is immediate to see that by choosing a sufficiently large $q\ge0$, the coefficient $z_0(q)$ can be made positive. When additionally $k_1\gamma_p-k_2>0$ holds, then $z_1(q)$ can also be made positive in the same way. Therefore, when these conditions on the controller gains hold, $z_0(q)$ and $z_1(q)$ both admit positive values provided that $q\ge0$ is chosen sufficiently large, proving then that the equilibrium point (\ref{eq:mupt}) is globally stable.

To prove global asymptotic stability, it is enough to note that since $u^*>0$, we have $\varphi(u^*)=u^*$ and the control input equilibrium value does not lie in the kernel of $\varphi$. According to \cite{Fliegner:06}, this allows to conclude on the global asymptotic stability of the equilibrium point (\ref{eq:mupt}). The proof is complete.
\end{proof}

It is immediate to obtain the following extension to the uncertain case:
\begin{lemma}
   Given system parameters $(k_p,\gamma_p,\gamma_r)\in\mathcal{P}_\mu$, then the equilibrium point  (\ref{eq:mupt}) of the closed-loop system (\ref{eq:musyst})-(\ref{eq:pi1}) is globally robustly asymptotically stable if the following conditions
   \begin{equation}\label{eq:djsqdj2}
   \begin{array}{rclcrcl}
     k_2&>&0&\text{and}& k_1&>&\dfrac{k_2}{\gamma_p^-}
   \end{array}
   \end{equation}
   hold.\mendth
\end{lemma}

It is interesting to note that the above results are also valid for any static nonlinearity in the sector $[0,1]$, including for instance saturations. It can also be easily adapted to consider nonlinear functions inside more general sectors, albeit the resulting conditions may be different.

Note that complementary sufficient conditions for nominal and robust stability could have been extracted from the polynomial $Z$ in (\ref{eq:P}), either by considering the cases  $k_1\gamma_p-k_2=0$ and  $k_1\gamma_p-k_2<0$; or by using Sturm series \cite{Sturm:29} that could provide necessary and sufficient conditions for the polynomial $Z$ to be positive over $\bar{\omega}\in[0,\infty)$. These conditions are however quite intricate and, for the sake of simplicity, have not been retained in the current work. It is indeed quite to difficult to draw interesting conclusions from them, unlike conditions (\ref{eq:djsqdj}) which can interpreted as a restriction of the local stability conditions (\ref{eq:nomcond}).

\subsection{Disturbance rejection}

It seems important to discuss disturbance rejection properties of the closed-loop system. Due to the nonlinear term, the rejection of constant input disturbances is only possible when they remain within certain bounds. This is formalized below:
\begin{lemma}
  Given system parameters $k_p,\gamma_p,\gamma_r>0$, the control law (\ref{eq:pi1}) globally rejects constant input disturbances $\delta_u$ that satisfy
  \begin{equation}\label{eq:conddist1}
    \delta_u\le\dfrac{\gamma_p\gamma_r}{k_p}\mu_*
  \end{equation}
  provided that the controller gains satisfy conditions (\ref{eq:djsqdj}).\mendth
\end{lemma}
\begin{proof}
  In presence of constant input disturbances, the equilibrium value of the control input is given by
  \begin{equation}
    u_\delta^*:=\dfrac{\gamma_p\gamma_r}{k_p}\mu_*-\delta_u.
  \end{equation}
  This value needs to be nonnegative in order to be driven by the on-off nonlinearity $\varphi$, which is the case if and only if condition (\ref{eq:conddist1}) holds.
\end{proof}

The above result readily extends to the uncertain case:
\begin{lemma}
  Assume $(k_p,\gamma_p,\gamma_r)\in\mathcal{P}_\mu$, the control law (\ref{eq:pi1}) satisfying conditions (\ref{eq:djsqdj2}) globally and robustly rejects constant input disturbances if and only if the condition
  \begin{equation}\label{eq:conddist2}
    \delta_u\le\dfrac{\gamma_p^+\gamma_r^+}{k_p^-}\mu_*
  \end{equation}
  is fulfilled.\mendth
\end{lemma}
\begin{proof}
  The proof follows from a simple extremum argument similar to the one used in the proof of Lemma \ref{lem:robloc}.
\end{proof}

It seems important to point out that the sets of admissible perturbations defined by (\ref{eq:conddist1}) or (\ref{eq:conddist2}) do not depend on the choice for the controller gains. They do however depend on the mean reference value $\mu_*$, which is expected since small $\mu_*$'s yield small control inputs that are more likely to be overwhelmed by disturbances.

\subsection{Concluding remarks}

The coefficient of variation $C_\nu:=\sigma_*/\mu_*$, defined as the ratio of the equilibrium values for the standard deviation and the mean number of proteins, is given in the current setup by
\begin{equation}
  C_\nu=\dfrac{1}{\sqrt{\mu_*}}\sqrt{1+\dfrac{k_p}{\gamma_p+\gamma_r}}.
\end{equation}
This shows that the equilibrium standard deviation depends on the desired mean value, and thus we have no control over it. Since the variance automatically increases as the mean increases, this motivates the aim of controlling the variance in order to keep it a reasonably low level.

\section{Mean and Variance control}\label{sec:var}

As discussed in the previous section, acting on $k_r$ is not sufficient for controlling both the mean and variance equilibrium values. It is shown in this section that variance control can be achieved by adding the second control input $\gamma_r\equiv u_2$. Fundamental limitations of the control system are discussed first, then local stabilizability is addressed.

\subsection{Fundamental limitations}

Let us consider in this section the control inputs $k_r\equiv u_1$ and $\gamma_r\equiv u_2$. It is shown below that there is a fundamental limitation on the references values for the mean and variance.
\begin{proposition}
  The set of admissible reference values $(\mu_*,\sigma_*^2)$ is given by the open and nonempty set
  \begin{equation}
    \mathcal{A}:=\left\{(x,y)\in\mathbb{R}_{>0}^2:\ x<y< \left(1+\dfrac{k_p}{\gamma_p}\right)x\right\}
  \end{equation}
  where $k_p,\gamma_p>0$. \mendth
\end{proposition}
\begin{proof}
  The lower bound is imposed by the coefficient of variation which gives
  \begin{equation}
  \sigma_*^2=\left(1+\dfrac{k_p}{\gamma_r+\gamma_p}\right)\mu_*>\mu_*.
\end{equation}
The upper bound is imposed by the positivity of the unique equilibrium control inputs values given by
\begin{equation}
  \begin{array}{lcl}
    u_1^*&=&\dfrac{\gamma_p}{k_p}\mu_*u_2^*,\\
    u_2^*&=&-\gamma_p+\dfrac{k_p\mu_*}{\sigma^2_*-\mu_*}
  \end{array}
\end{equation}
which are well-posed since $\sigma^2_*-\mu_*>0$ according to the coefficient of variation constraint. The second equilibrium control input value $u_2^*$ is positive if and only if ${\sigma_*^2< \left(1+\dfrac{k_p}{\gamma_p}\right)\mu_*}$, which in turn implies that $u_1^*$ is nonnegative as well. The proof is complete.
\end{proof}
The lower bound obtained above remains valid when $k_p$ or $\gamma_p$ are chosen as second control inputs. The factor of the upper-bound however changes to $1+k_p/\gamma_r$ when $\gamma_p\equiv u_2$, or becomes unconstrained when $k_p\equiv u_2$. Note however that the upper bound on the variance is not a strong limitation in itself because we are mostly interested in achieving low variance.

Note also that since the lower bound on the achievable variance is independent of the controller structure, it is hence pointless to look for advanced control techniques in view of improving this limit. A positive fact, however, is that the lower bound is fixed and does not depend on the knowledge of the parameters of the system. This potentially makes low equilibrium variance robustly achievable.

\subsection{Problem formulation}

Considering the control inputs $k_r\equiv u_1$ and $\gamma_r\equiv u_2$, the system (\ref{eq:mainsyst}) can be rewritten as the bilinear system
\begin{equation}\label{eq:bilsyst}
  \begin{array}{lcl}
    \dot{x}_1&=&-u_2x_1+u_1\\
    \dot{x}_2&=&k_px_1-\gamma_px_2\\
    \dot{x}_3&=&u_2x_1-2u_2x_3+u_1\\
    \dot{x}_4&=&k_px_3-\gamma_px_4-u_2x_4\\
    \dot{x}_5&=&k_px_1+\gamma_px_2+2k_px_4-2\gamma_px_5\\
    \dot{I}_1&=&\mu_*-x_2\\
    \dot{I}_2&=&\sigma^2_*-x_5\\
  \end{array}
\end{equation}
where $I_1$ and $I_2$ are the states of the integrators. The control inputs are defined as the outputs of a multivariable positive PI controller
\begin{equation}\label{eq:clb}
  \begin{array}{lcl}
    u_1&=&\varphi\left(k_1e_1+k_2I_1+k_3e_2+k_4I_2\right)\\
    u_2&=&\varphi\left(k_5e_1+k_6I_1+k_7e_2+k_8I_2\right)
  \end{array}
\end{equation}
where $e_1:=\mu_*-x_2$ and $e_2:=\sigma_*^2-x_5$.

\begin{property}
Assume that $k_2k_8-k_4k_6\ne0$, then the equilibrium point of the system (\ref{eq:bilsyst})-(\ref{eq:clb}) is unique and given by
\begin{equation}\label{eq:eqptv1}
  \begin{array}{l}
    x_1^*=\dfrac{\gamma_p}{k_p}\mu_*,\quad x_2^*=\mu_*,\quad  x_3^*=x_1^*,\quad x_4^*=\dfrac{\gamma_p}{\gamma_p+u_2^*}\mu_*,\vspace{3mm}\\
    x_5^*=\sigma_*^2,\quad u_1^*=\dfrac{\gamma_p}{k_p}\mu_*u_2^*,\quad  u_2^*=-\gamma_p+\dfrac{k_p\mu_*}{\sigma^2_*-\mu_*}
  \end{array}
\end{equation}
and
\begin{equation}\label{eq:eqptv2}
  \left[\begin{array}{c}
    I_1^*\\
    I_2^*
  \end{array}\right]=\left[\begin{array}{cc}
    k_2 & k_4\\
    k_6 & k_8
  \end{array}\right]^{-1}\left[\begin{array}{c}
    u_1^*\\
    u_2^*
  \end{array}\right].
\end{equation}
\end{property}

\vspace{2mm}Associated with the set of admissible references $\mathcal{A}$, we define the set of equilibrium points as
\begin{equation}
  \mathcal{X}^*:=\left\{(x^*,I^*)\in\mathbb{R}^7:\ (y_*,\sigma_*^2)\in\mathcal{A}\right\}.
\end{equation}

\subsection{Local stabilizability and stabilization}

Since the equilibrium control inputs are positive, the nonlinearities are not active in a neighborhood of the equilibrium point (\ref{eq:eqptv1})-(\ref{eq:eqptv2}). Local analysis can hence be performed using standard linearization techniques. The corresponding  Jacobian system is given by
\begin{equation}
  \dot{x}_\ell=A^*_\ell x_\ell
\end{equation}
where $A^*_\ell$ is given in (\ref{eq:As}) with $\delta:=\mu_*-\sigma_*^2$.

The following result states conditions for the Jacobian system to be locally representative of the behavior of the original nonlinear system:
\begin{lemma}
 The Jacobian system fully characterizes the local behavior of the controlled nonlinear system (\ref{eq:bilsyst})-(\ref{eq:clb}) if and only if the condition $k_2k_8-k_4k_6\ne0$ holds.\mendth
\end{lemma}
\begin{proof}
  For the Jacobian system to represent the local behavior, it is necessary and sufficient that $A^*_\ell$ has no eigenvalue at 0. A quick check at the determinant value
  \begin{equation*}
    \det(A^*_\ell)=4\gamma_pk_p(k_2k_8-k_4k_6)(\mu_*(k_p+\gamma_p)-\gamma_p\sigma_*^2)
  \end{equation*}
  yields that the condition $k_2k_8-k_4k_6\ne0$ is necessary and sufficient for the local representativity of the nonlinear system. Note that since $(\mu_*,\sigma_*^2)\in\mathcal{A}$, the term ${\mu_*(k_p+\gamma_p)-\gamma_p\sigma_*^2}$ is always different from 0. The proof is complete.
\end{proof}
\begin{figure*}
  \begin{equation}\label{eq:As}
A^*_\ell=\begin{bmatrix}
    \gamma_p+\frac{k_p\mu_*}{\delta} & -k_1+\frac{\gamma_pk_5\mu_*}{k_p} & 0 & 0 & -k_3+\frac{\gamma_pk_7\mu_*}{k_p} & k_2-\frac{\gamma_pk_6\mu_*}{k_p} & k_4-\frac{\gamma_pk_8\mu_*}{k_p}\\
    k_p & -\gamma_p & 0 & 0 & 0 & 0 & 0\\
    -\gamma_p-\frac{k_p\mu_*}{\delta} &  -k_1+\frac{\gamma_pk_5\mu_*}{k_p} & 2\gamma_p+2\frac{k_p\mu_*}{\delta} & 0 & -k_3+\frac{\gamma_pk_7\mu_*}{k_p} & k_2-\frac{\gamma_pk_6\mu_*}{k_p} & k_4-\frac{\gamma_pk_8\mu_*}{k_p}\\
    0 & -k_5\frac{\gamma_p\delta}{k_p} & k_p & \frac{k_p\mu_*}{\delta} & -\frac{\gamma_pk_7\delta}{k_p} & \frac{\gamma_pk_6\delta}{k_p} & \frac{\gamma_pk_8\delta}{k_p}\\
    k_p & \gamma_p & 0 & 2k_p & -2\gamma_p & 0 & 0\\
    0 & -1 & 0 & 0 & 0& 0 & 0\\
    0 & 0 & 0 & 0 & -1& 0 & 0
  \end{bmatrix}
\end{equation}
 \hrulefill
\end{figure*}

The local system being linear, the Routh-Hurwitz criterion could have indeed be applied as in the mean control case, but would have led to very complex algebraic inequalities, difficult to analyze in the general case, even for simple controller structures. The Popov-Belevitch-Hautus (PBH) stabilizability test would not have helped either to conclude on anything useful since it does not take into account the controller structure. Despite the `large size' of the matrix $A^*_\ell$, it is fortunately still possible to provide a stabilizability result using the fact that $A^*_\ell$ is marginally stable\footnote{$A_\ell^*$ has eigenvalues in the closed-left half plane and those on the imaginary axis are semisimple \cite{Meyer:00}.} when the control parameters $k_i$ are set to 0. This is obtained using perturbation theory of nonsymmetric matrices \cite{Seyranian:03}.
\begin{lemma}
  Given any $k_p,\gamma_p>0$, the bilinear system (\ref{eq:bilsyst}) is locally asymptotically stabilizable around any equilibrium point (\ref{eq:eqptv1})-(\ref{eq:eqptv2}) using the control law (\ref{eq:clb}).\mendth
\end{lemma}
\begin{proof}
The perturbation argument relies on checking whether the eigenvalues on the imaginary axis can be shifted by slightly perturbing the controller coefficients around the `0-controller', i.e. by letting $k_i=\eps d_i$, where $\eps\ge0$ is the small perturbation parameter and $d_i$ is the perturbation direction for controller parameter $k_i$. We assume here that both integrators are involved in the controller, that is $|d_2|+|d_6|>0$ and $|d_4|+|d_8|>0$. To prove the result, let us first rewrite the matrix $A^*_\ell$ as
\begin{equation}
  A^*_\ell=A_0+\eps\sum_{j=1}^8d_jA_j.
\end{equation}
The matrix $A_0$ is a marginally stable matrix with a semisimple eigenvalue of multiplicity two at zero. Paradoxically, these eigenvalues introduced by the PI controller are the only critical ones that must be stabilized, i.e. shifted to the open left-half plane. From perturbation theory of general matrices \cite{Seyranian:03}, it is known that semisimple eigenvalues bifurcate into (distinct or not) eigenvalues according to the expression \cite{Seyranian:03}
\begin{equation}\label{eq:eigexp}
  \lambda_i(\eps,d)=\eps\xi_i(d)+o(\eps),\ i=1,2
\end{equation}
where $\xi_i(d)$ is the $i^{th}$ eigenvalue of the matrix ${M(d):=\sum_{i=1}^8d_iM_i}$ with
\begin{equation}
M_i:=\begin{bmatrix}
    \nu_\ell^1\\
    \nu_\ell^2
  \end{bmatrix}A_j\begin{bmatrix}
   \nu_r^1&& \nu_r^1
  \end{bmatrix},\ i=1,\ldots,8.
\end{equation}
Above, $\nu_\ell^1,\nu_\ell^2$ and $\nu_r^1,\nu_r^2$ are the normalized\footnote{Normalized eigenvectors verify the conditions $\nu_\ell^1\nu_r^1=\nu_\ell^2\nu_r^2=1$ and ${\nu_\ell^1\nu_r^2=\nu_\ell^2\nu_r^1=0}$.} left- and right-eigenvectors associated with the semisimple zero eigenvalue. It turns out that all $M_i$'s with odd index are zero, indicating that the proportional gains have a locally negligible stabilizing effect. This hence reduces the size of the problem to 4 parameters, i.e. those related to integral terms. We make now the additional restriction that $d_4=d_6=0$ reducing the controller structure to one integrator per control channel. The matrix $M(d)$ then becomes
\begin{equation*}
  M(d)=\psi\begin{bmatrix}
        \dfrac{k_pd_2}{\gamma_p} & -d_8\mu_*\\
        \dfrac{k_p\sigma_*^2d_2}{\gamma_p\mu_*} & \dfrac{\gamma_p(\mu_*-\sigma_*^2)^2}{k_p\mu_*}+\mu_*-2\sigma_*^2
  \end{bmatrix}
\end{equation*}
where $\psi:=\dfrac{\mu_*-\sigma_*^2}{\gamma_p(\mu_*-\sigma_*^2)+k_p\mu_*}$. The semisimple eigenvalues then move to the open left-half plane if there exist perturbation directions $d_2,d_8\in\mathbb{R}$, $d_2d_8\ne0$, such that $M(d)$ is Hurwitz. We can now invoke the Routh-Hurwitz criterion on $M(d)$ and we get the conditions
\begin{equation*}
  \begin{array}{rcl}
    d_2d_8\psi\dfrac{(\mu_*-\sigma_*^2)^2}{\gamma_p\mu_*}&>&0\\
    \gamma_p\psi\left(d_8\left(\gamma_p(\mu_*-\sigma_*^2)^2-2k_p\mu_*\sigma_*^2\right)+d_2k_p\mu_*\sigma_*^2\right)&<&0.
  \end{array}
\end{equation*}
Since the term $\psi$ is negative for all $(\mu_*,\sigma_*^2)\in\mathcal{A}$, the first inequality holds true if and only if $d_2d_8<0$, i.e. perturbation directions have different signs. The second inequality can be rewritten as
\begin{equation}\label{eq:lastcond}
  d_2>d_8\left(2-\dfrac{\gamma_p(\mu_*-\sigma_*^2)^2}{k_p\mu_*\sigma_*^2}\right).
\end{equation}
Choosing then $d_8<0$, there always exists $d_2>0$ such that the above inequality is satisfied, making thus the matrix $M(d)$ Hurwitz. We have hence proved that for any given pair $(\mu_*,\sigma_*^2)\in\mathcal{A}$, there exists a control law (\ref{eq:clb}) that makes the corresponding equilibrium locally asymptotically stable. The proof is complete.
\end{proof}

It is possible to go beyond this result and show that there exist semi-global PI controllers:
\begin{lemma}
 Given any $k_p,\gamma_p>0$, there exists a common control law (\ref{eq:clb}) that simultaneously locally asymptotically stabilizes system  (\ref{eq:bilsyst}) around all the equilibrium points in $\mathcal{X}^*$.\mendth
\end{lemma}
\begin{proof}
 To show that there exists a common controller that simultaneously makes all the equilibrium points in $\mathcal{X}^*$ locally asymptotically stable, it is enough to prove that there exists a pair $(d_2,d_8)\in\mathbb{R}^2$, $d_2d_8\ne0$ such that the inequality (\ref{eq:lastcond}) is satisfied for all $(\mu_*,\sigma_*^2)\in\mathcal{A}$. This is equivalent to finding a finite $d_2>0$ satisfying
\begin{equation}
  d_2>d_8\left(2-\sup_{(\mu_*,\sigma_*^2)\in\mathcal{A}}\left\{\dfrac{\gamma_p(\mu_*-\sigma_*^2)^2}{k_p\mu_*\sigma_*^2}\right\}\right).
\end{equation}
Standard analysis allows to prove that
\begin{equation}
  \sup_{(\mu_*,\sigma_*^2)\in\mathcal{A}}\left\{\dfrac{\gamma_p(\mu_*-\sigma_*^2)^2}{k_p\mu_*\sigma_*^2}\right\}=\dfrac{k_p}{\gamma_p+k_p}\in(0,1)
\end{equation}
which shows that by simply choosing the directions $d_8<0$ and $d_2>0$, the matrix $M(d)$ becomes then Hurwitz for all $(\mu_*,\sigma_*^2)\in\mathcal{A}$. This therefore implies the existence of a common control law (\ref{eq:clb}) that locally and asymptotically simultaneously stabilizes all the equilibrium points in $\mathcal{X}^*$. 
\end{proof}


\section{Examples}\label{sec:ex}

For simulation purposes, we consider the normalized version of system (\ref{eq:bilsyst}) similarly as in \cite{Khammash:11}:
\begin{equation*}
  \begin{array}{lcl}
    \dot{\bar{x}}_1(t)&=&-(\gamma_r^0+u_2)\bar{x}_1(t)+\tilde{u}_1(t)\\
    \dot{\bar{x}}_2(t)&=&\gamma_p(\bar{x}_1(t)-\bar{x}_2(t))\\
    \dot{\bar{x}}_3(t)&=&(\gamma_r^0+u_2(t))\bar{x}_1(t)-2(\gamma_r^0+u_2)\bar{x}_3(t)+\tilde{u}_1(t)\\
    \dot{\bar{x}}_4(t)&=&(\gamma_r^0+\gamma_p)\bar{x}_3(t)-(\gamma_r^0+u_2+\gamma_p)\bar{x}_4(t)\\
    \dot{\bar{x}}_5(t)&=&\dfrac{\gamma_p}{\alpha}\left[\bar{x}_1(t)+\bar{x}_2(t)+2(\alpha-1)\bar{x}_4(t)\right]-2\gamma_p\bar{x}_5(t)
  \end{array}
\end{equation*}
where $\tilde{u}_1(t)=\gamma_r^0+bu_1(t)$, $\gamma_r^0=0.03$, $\gamma_p=0.0066$, ${b=0.9587}$, $k_p=0.06$ and $\alpha=1+k_p/(\gamma_r^0+\gamma_p)$. The system has been normalized according to basal levels for transcription rate $k_r^0$ and degradation rate $\gamma_r^0$. In the absence of control inputs, i.e. $\tilde{u}_1\equiv0$ and $u_2\equiv0$, the system converges to the normalized equilibrium values $\bar{x}_i^*=1$, $i=1,\ldots,5$. The parameter values are borrowed from \cite{Khammash:11}.

\subsection{Mean control}

The considered PI controller parameters computed using loop shaping are $k_1=0.01$ and $k_2=0.0007$. Simulations yield the trajectories of Fig. \ref{fig:mean1} and Fig. \ref{fig:mean2}. We can see that, as expected, the proposed controller achieves output tracking for different references and in presence of constant input disturbances.

\begin{figure}[h]
  \hspace{-9mm}\includegraphics[width=0.55\textwidth]{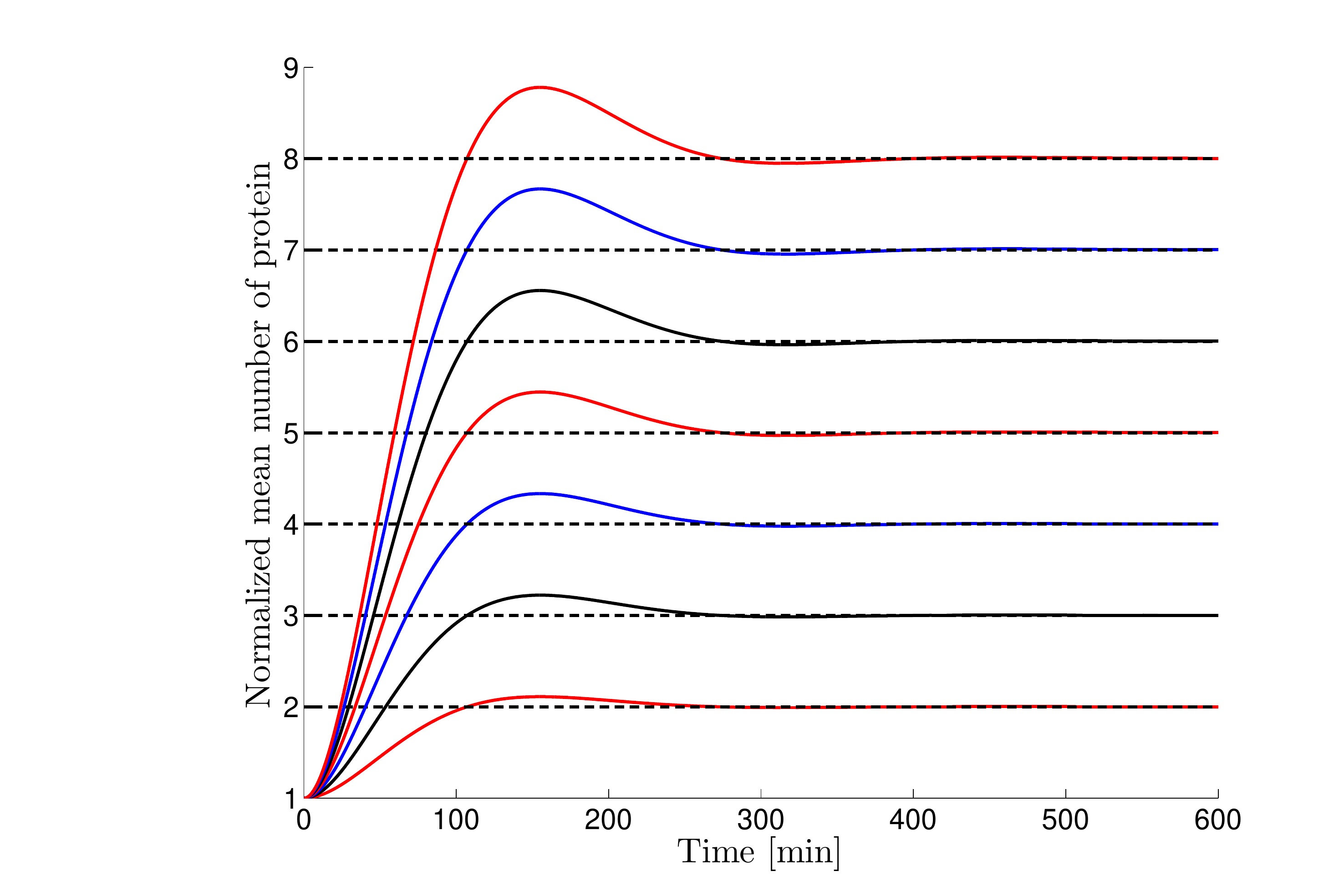}
\caption{Trajectories of the mean number of proteins for different reference values.}\label{fig:mean1}
\end{figure}

\begin{figure}[h]
  \hspace{-9mm}\includegraphics[width=0.55\textwidth]{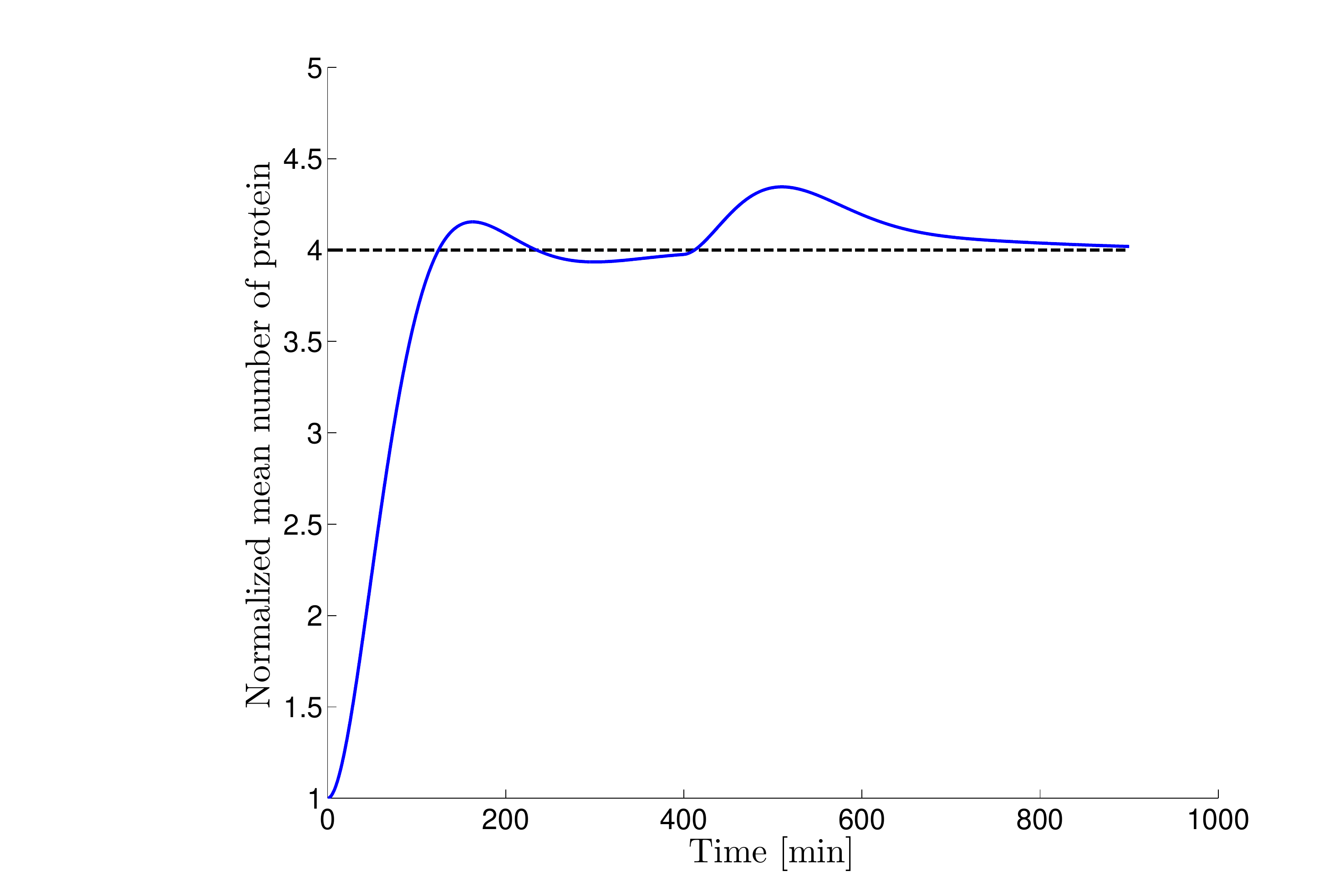}
\caption{Trajectories of the mean number of proteins in response to constant input disturbance.}\label{fig:mean2}
\end{figure}


%


\subsection{Mean and variance control}

In this case, a PI controller with gain $k_1=1$, $k_2=0.007$, $k_3=-0.2$ and $k_4=-0.0014$  is considered.  The normalized achievable minimal variance is given by
\begin{equation}
  \sigma_{min}^2=\dfrac{\gamma_r^0+\gamma_p}{\gamma_r^0+\gamma_p+k_p}\mu.
\end{equation}
The response of the controlled variance according to changes in the reference value is depicted in Fig. \ref{fig:var1} where we can see that the variance tracks the desired value quite well. In order to avoid oscillations, the changes in the reference values follow a ramp. It seems also important to point out that when the reference point changes, due to the coupling between the mean and variance, the mean value changes as well, but this is immediately corrected by the mean controller.

\begin{figure}[h]
  \hspace{-9mm}\includegraphics[width=0.55\textwidth]{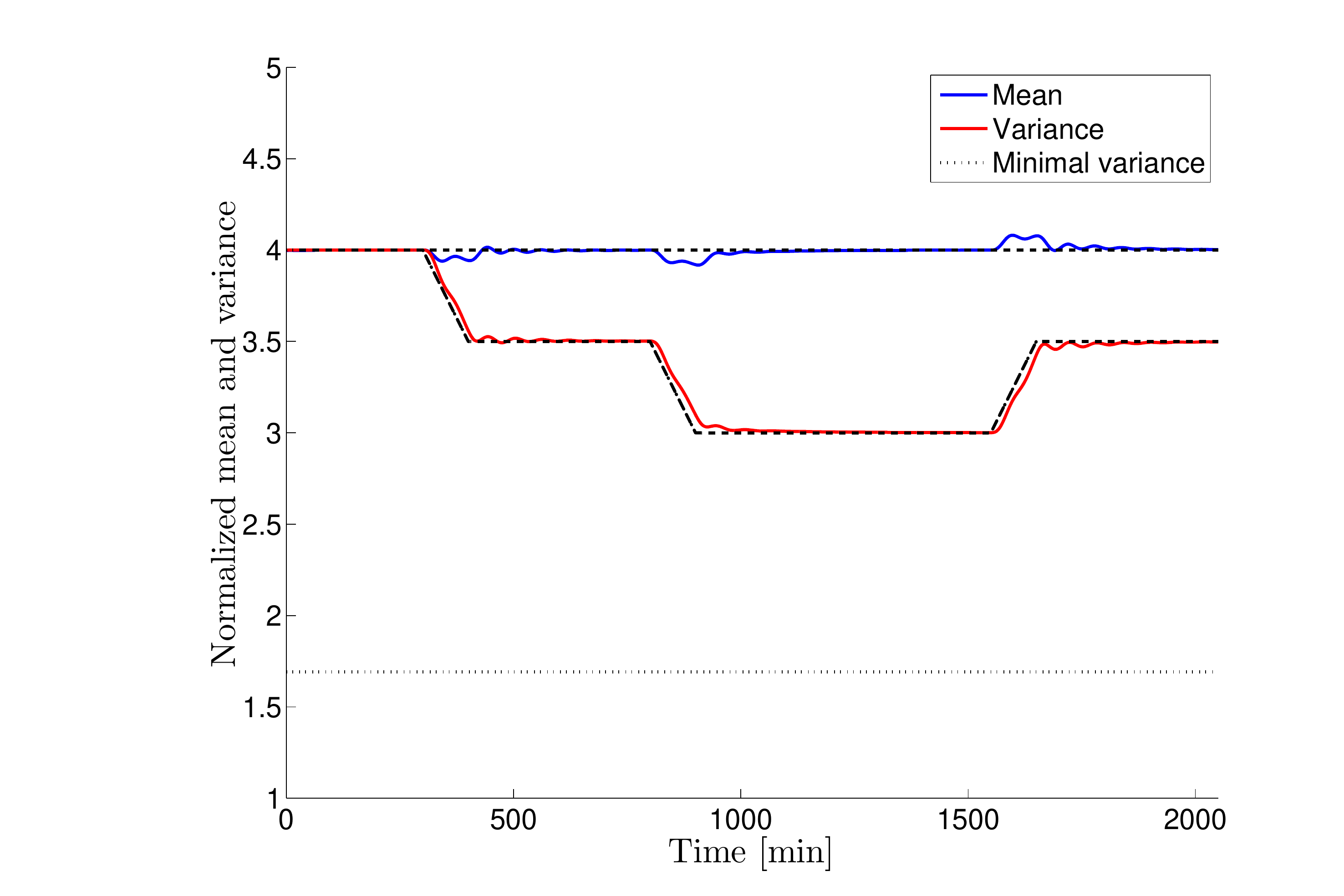}
\caption{Response of the controlled variance according to changes in the reference.}\label{fig:var1}
\end{figure}

\section{Conclusion}

Controlling the mean and variance of the number of proteins in a simple gene expression circuit using PI controllers has been shown to be achievable. Interestingly, PI controllers have been proved to be sufficient for respecting most of the inner (system) and outer (design) constraints, and to be global in the sense that a single PI can locally stabilize all the possible equilibrium points. The very same results have been extended to uncertain systems.

Future works will be devoted to a better characterization of the stability domain of the controlled variance dynamics (invariant set, global stability, etc), the derivation of advanced PI controllers to improve transient behavior and addressing implementation issues. The generalization of this idea to the non-affine propensity case and the control of higher order moments are also important problems that will be considered. Finally, the implementation on the real process is under way.

\bibliographystyle{ieeetran}

\end{document}